\documentclass[psamsfonts]{amsart}

\usepackage{amssymb,amsfonts}
\usepackage[all,arc]{xy}
\usepackage{enumerate}
\usepackage{mathrsfs}
\usepackage{slashbox}
\usepackage{geometry}
\usepackage{hyperref}

\newtheorem{theorem}{Theorem}[section]

\newtheorem{lemma}[theorem]{Lemma}

\theoremstyle{definition}
\newtheorem{definition}[theorem]{Definition}
\newtheorem{remark}{Remark}

\numberwithin{equation}{section}

\title[3D NAVIER-STOKES-KORTEWEG]{A blow-up criterion for the strong solutions to the nonhomogeneous Navier-Stokes-Korteweg equations in dimension three}

\date{\today}
\email{lihuanyuan1111@163.com}
\begin{document}
\maketitle

\centerline{\scshape Huanyuan Li}
\medskip
{\footnotesize
 \centerline{School of Mathematics and Statistics, Zhengzhou University,}
   \centerline{Zhengzhou, 450001, People's Republic of China}

} 


\begin{abstract}
This paper proves a Serrin's type blow-up criterion for the 3D density-dependent Navier-Stokes-Korteweg equations with vacuum. It is shown that if the density $\rho$ and velocity field $u$ satisfy  $\|\nabla \rho\|_{L^{\infty}(0,T; W^{1,q})}$ $ + \| u\|_{L^s(0,T; L^r_{\omega})}$ $< \infty$ for some $q>3$, and any $(r,s)$ satisfying $\frac{2}{s}+\frac{3}{r} \le 1,~~3 <r \le \infty,$ then the strong solutions to the density-dependent Navier-Stokes-Korteweg equations can exist globally over $[0,T]$, here $L^r_{\omega}$ denotes the weak $L^r$ space.

\end{abstract}

\section{Introduction and main result}

It is well-known that some known mathematical results on the homogeneous incompressible Navier-Stokes equations between the dimension three and two are very different. For example, the global well-posedness of the two-dimensional incompressible Navier-Stokes equations has been proved long time ago, however, the three-dimensional global well-posedness for large initial data is still a famous open problem in the partial differential equations. And we believe that the similar dimensional differences also appear in the analysis of the nonhomogeneous fluid dynamics.This is a continuous work of {\cite{Li}}, in which the author established a blow-up criterion for the strong solutions to the initial and boundary value problem of the nonhomogeneous incompressible Navier-Stokes-Korteweg equations in dimension two. And the purpose of this paper is to establish a blow-up criterion for the strong solutions to the initial and boundary value problem of the nonhomogeneous incompressible Navier-Stokes-Korteweg equations in dimension three, which will involve not only the density but also the velocity field. And our result also indicates the famous Serrin's criterion for the classical (homogeneous) Navier-Stokes equations.

The time evolution of the density $\rho = \rho(x, t)$, velocity field $u = (u_1, u_2, u_3)(x, t) $ and pressure $P = P(x, t)$ of a general viscous capillary fluid is governed by the nonhomogeneous incompressible Navier-Stokes-Korteweg equations
\begin{equation} \label{NSK}
\left\{
\begin{aligned}
&\partial_t\rho + \mathrm{div}(\rho u) = 0, \\
&\partial_t(\rho u) + \mathrm{div} (\rho u\otimes u) - \mathrm{div}(2\mu(\rho)d) +\nabla P+ \mathrm{div}(\kappa(\rho)\nabla \rho\otimes \nabla \rho) = 0, \\
& \mathrm{div} u=0, \\
\end{aligned}
\right.
\end{equation}
where $x \in \Omega $ is the spatial coordinate, and $t \ge 0$ is the time. In this paper, $\Omega$ is a bounded domain with smooth boundary in $\mathbb{R}^3$. 
$$ d = \frac{1}{2} \Big[ \nabla u + (\nabla u)^T \Big] $$
denotes the deformation tensor of the matrix form with the $ij$ component $\frac{1}{2} (\partial u_i / \partial x_j + \partial u_i / \partial x_j  )$. $\kappa=\kappa(\rho)$, which is a $C^1$ nonnegative function of the density $\rho$, stands for the capillary coefficient. And $\mu= \mu(\rho)$ is the viscosity coefficient of the fluids, which is assumed to be a function of density $\rho$ satisfying
\begin{equation}\label{viscosity}
 \mu \in C^1[0,\infty),~~ \rm{and} ~~\mu \ge \underline{\mu} > 0 ~~ \rm{on} ~ [0,\infty)
\end{equation}
for some positive constant $\underline{\mu}$.

We focus on the system \eqref{NSK}-\eqref{viscosity} with the initial and boundary conditions:
\begin{equation} \label{initial}
u =0, \quad \mathrm{on}~ \partial \Omega \times [0,T),
\end{equation}
\begin{equation} \label{bdy}
(\rho,u)|_{t=0} = (\rho_0,u_0) \quad \mathrm{in}~\Omega.
\end{equation}

When $\kappa \equiv 0$, the system \eqref{NSK}-\eqref{bdy} are the famous nonhomogeneous incompressible Navier-Stokes equations with density-dependent viscosity. Cho and Kim {\cite{ChoY}} proved the local existence of unique strong solution for all initial data satisfying a compatibility condition. And later Huang and Wang 
{\cite{HW3}} proved  the strong solution exists globally in time when the initial gradient of the velocity is suitably small. For the related progress,  see {\cite{HW1}-{\cite{HW3}} references and therein.

Let us come back to the fluids with capillary effect, that is, $\kappa(\rho)$ depends on the density $\rho$. As far as I know, the first local existence of unique strong solution was obtained by Tan and Wang {\cite{TW}} when the capillary coefficients $\kappa$ is a nonnegative constant. And very recently, Wang {\cite{W}} extended their result to the case when $\kappa(\rho)$ is a $C^1$ function of the density. 

First we give the definition of strong solutions to the initial and boundary problem  \eqref{NSK}-\eqref{bdy} as follows.

\begin{definition}[Strong solutions]
A pair of functions $(\rho \ge 0, u, P)$ is called a strong solution to the problem \eqref{NSK}-\eqref{bdy} in $\Omega \times (0,T)$, if for some $q_0 \in (3, 6]$, 
\begin{equation} \label{solreg1}
\begin{aligned}
&\rho \in C([0,T];W^{2,q_0}), \quad  u \in C([0,T];H^1_0\cap H^2), \quad \nabla^2 u \in L^2(0,T;L^{q_0}), \\
&\rho_t \in C([0,T];W^{1,q_0}), \quad  \nabla P \in C([0,T]; L^2)\cap L^2(0,T;L^{q_0}), \quad u_t \in L^2(0,T;H^1_0),\\
\end{aligned}
\end{equation}
and $(\rho, u, P)$ satisfies  \eqref{NSK} a.e. in $\Omega \times (0,T)$.
\end{definition}

In the case when the initial data may vanish in an open subset of $\Omega$, that is, the initial vacuum is allowed, the following local well-posedness of strong solution to  \eqref{NSK}-\eqref{bdy} was obtained by Wang {\cite{W}}.
\begin{theorem}  \label{local}
Assume that the initial data $(\rho_0,u_0)$ satisfies the regularity condition
\begin{equation} \label{initialreg}
0 \le \rho_0 \in W^{2,q},\quad 3<q \le 6, \quad u_0 \in H^{1}_{0,\sigma} \cap H^2,
\end{equation}
and the compatibility condition
\begin{equation} \label{initialcompa}
-\mathrm{div}(\mu(\rho_0) (\nabla u_0 + (\nabla u_0)^T))  + \nabla P_0+ \mathrm{div}(\kappa(\rho_0)\nabla \rho_0\otimes \nabla \rho_0) = \rho_0^{1/2}g,
\end{equation}
for some $(P_0,g) \in H^1 \times L^2$. Then  there exist a small time  $T$ and a unique strong solution $(\rho,u,P) $ to the initial boundary value problem \eqref{NSK}-\eqref{bdy}.
\end{theorem}

Motivated by the work of Kim {\cite{K}}, in which a Serrin's type blow-up criterion for the 3D nonhomogeneous incompressible Navier-Stokes flow was established, we derive a similar blow-up criterion for the nonhomogeneous Navier-Stokes-Korteweg equations with density-dependent viscosity and capillary coefficients in dimension three. More precisely, our main result can be stated as follows.
\begin{theorem} \label{thm}
Assume that the initial data $(\rho_0,u_0)$ satisfies the regularity condition \eqref{initialreg} and
the compatibility condition \eqref{initialcompa}. Let $(\rho, u, P)$ be a strong solution of the problem \eqref{NSK}-\eqref{bdy} satisfying \eqref{solreg1}. If $0<T^*<\infty$ is the maximal time of existence, then
\begin{equation} \label{blowcriterion}
\lim_{T \rightarrow T*}(\|\nabla \rho\|_{L^{\infty}(0,T; W^{1,q})} + \| u\|_{L^s(0,T; L^r_{\omega})}) = \infty.
\end{equation}
for any $r$ and $s$ satisfying 
\begin{equation}  \label{serrincomponent}
\frac{2}{s}+\frac{3}{r} \le 1, \quad 3 <r \le \infty,
\end{equation}
where $L^r_{\omega}$ denotes the weak $L^r$ space.
\end{theorem}

\begin{remark}
Compared to the two-dimensional blow-up criterion established in \cite{Li} by the author, the blow-up criterion obtained in this paper involves not only the density but aslo the velocity field, see \eqref{blowcriterion}. And when $\rho_0 \equiv 1$, the nonhomogeneous incompressible Navier-Stokes-Korteweg equations \eqref{NSK} reduce to the classical incompressible Navier-Stokes equations, therefore our blow-up criterion indicates the generalization of Serrin's criterion using weak Lesbegue spaces for incompressible Navier-Stokes equations, see the work of H. Sohr (2001), S. Bosia et. al. (2014).

\end{remark}

The proof of Theorem \ref{thm} is based on the contradiction argument. In view of the local existence result, to prove Theorem \ref{thm}, it suffices to verify that $(\rho, u)$ satisfy \eqref{initialreg} and \eqref{initialcompa} at the time $T^*$ under the assumption of the left hand side of \eqref{blowcriterion} is finite. Unlike the Navier-Stokes equations treated in Kim {\cite{K}}, the use of weak Lesbegue space makes it more difficult to obtain some estimates because of the apperance of capillary effect. To overcome the difficulty,  we make good use of the finiteness of $\|\nabla \rho\|_{W^{1,q}}$ and other interpolation techniques in Lorentz space.

The remainder of this paper is arranged as follows. In Sec. 2, we give some auxiliary lemmas which is useful in our later analysis. The proof of Theorem \ref{thm} will be done by combining the contradiction argument with the estimates derived in Sec. 3.
\section{Preliminaries}

\subsection{Notations and general inequalities}

$\Omega$ is a bounded domain in $\mathbb{R}^3$ with smooth boundary $\partial \Omega$. For notations simplicity below, we omit the integration domain $\Omega$. And
for $1 \le r \le \infty$ and $ k \in \mathbb{N}$, the Lesbegue and Sobolev spaces are defined in a standard way,
$$ L^r = L^r (\Omega), \quad W^{k,r} = \{ f \in L^r : \nabla^k f \in L^r \}, \quad H^k = W^{k,2}.$$
The following Gagliardo-Nirenberg inequality will be used frequently in the later analysis.
\begin{lemma}[Gagliardo-Nirenberg inequality] \label{GNinequality}
Let $\Omega$ be a domain of $\mathbb{R}^3$ with smooth boundary $\partial \Omega$. For $p \in[2,6], q \in (1,\infty)$ and $r \in (3, \infty)$, there exists some generic constants $C >0$ that may depend on $q$ and $r$ such that for $f \in H^1$ satisfying $f|_{\partial \Omega}=0$, and $g \in L^q \cap D^{1,r}$, we have
\begin{equation}  \label{GNin1}
\|f\|_{L^p}^p \le C \|f\|_{L^2}^{(6-p)/2}\|\nabla f\|_{L^2}^{(3p-6)/2},
\end{equation}
\begin{equation} \label{GNin2}
\|g\|_{L^{\infty}} \le C \|g\|_{L^q}^{q(r-3)/(3r+q(r-3))}\|\nabla g\|_{L^r}^{3r/(3r+q(r-3))}.
\end{equation}
\end{lemma}
See the proof of this lemma in Ladyzhenskaya et al. {\cite[P. 62]{LaSoN}}. \\
Denote the Lorentz space and its norm by $L^{p,q}$ and $\|\cdot\|_{L^{p,q}}$, respectively, where $1<p< \infty$ and $ 1 \le q \le \infty$. And we recall the weak-$L^p$ space $L^p_{\omega}$ which is defined as follows:
$$ L^p_{\omega} := \{ f \in L^1_{loc}: \|f \|_{L^p_{\omega}} = \sup_{\lambda>0} \lambda|\{|f(x)| > \lambda\}|^{\frac{1}{p}} < \infty \}.$$
And it should be noted that
$$ L^p \subsetneqq L^p_{\omega}, \quad L^{\infty}_{\omega} = L^{\infty}, \quad L^p_{\omega} = L^{p,\infty}.$$
For the details of Lorentz space, we refer to the first chapter in Grafakos {\cite{Gra}}.
The following lemma involving the weak Lesbegue spaces has been proved in Kim {\cite{K}}, Xu and Zhang {\cite{XW}}, which will play an important role in the subsequent analysis.
\begin{lemma} \label{weaksobolev}
Assume $g \in H^1$, and $f \in L^r_{\omega}$ with $r\in (3, \infty]$, then $f\cdot g \in L^2$. Furthermore, for any $\epsilon >0$, we have
\begin{equation}  \label{weakembed}
\|f\cdot g\|_{L^2}^2 \le \epsilon \|g\|_{H^1}^2 + C(\epsilon) (\|f\|_{L^r_{\omega}}^s +1) \|g\|_{L^2}^2,
\end{equation}
where $C$ is a positive constant depending only on $\epsilon, r$ and the domain $\Omega$.
\end{lemma}
\subsection{Higher order estimates on $u$}
High-order a priori estimates of velocity field $u$ rely on the following regularity results for density-dependent Stokes equations.
\begin{lemma} \label{stokesesti}
Assume that $\rho \in W^{2,q}, 3<q<\infty$, and $0 \le \rho \le \bar{\rho}$. Let $(u,P) \in H^{1}_{0,\sigma}  \times L^2$
be the unique weak solution to the boundary value problem
\begin{equation} \label{stokeseq}
-\mathrm{div} (\mu(\rho)(\nabla u +(\nabla u)^T) + \nabla P = F, \quad \mathrm{div} u = 0 ~\text{in}~\Omega, ~~~\text{and} \int P dx = 0,
\end{equation}
where
$$ \mu \in C^1[0,\infty), ~~\underline{\mu} \le \mu(\rho) \le \bar{\mu} ~ on ~[0,\bar{\rho}]. $$
Then we have the following regularity results:

(1) If $F \in L^2$, then $(u,P) \in H^2 \times H^1$ and
\begin{equation} \label{stokesesti2}
\| u \|_{H^2} + \| P \|_{H^1} \le C (1+\|\nabla \rho\|_{L^{\infty}})\|F\|_{L^2},
\end{equation}

(2) If $F \in L^r$ for some $r \in (2,\infty)$, then $(u,P) \in W^{2,r} \times W^{1,r}$ and
\begin{equation} \label{stokesestir}
\| u \|_{W^{2,r}} + \| P \|_{W^{1,r}} \le C (1+\|\nabla \rho\|_{L^{\infty}})\|F\|_{L^r}.
\end{equation}
\end{lemma}
The proof of Lemma \ref{stokesesti} has been given by Wang {\cite{W}}. And refer to Lemma 2.1 in his paper.

\section{Proof of Theorem 1.3}

Let $(\rho,u,P)$  be a strong solution to the initial and boundary value problem \eqref{NSK}-\eqref{bdy} as derived in Theorem \ref{local}. Then it follows from the standard energy estimate that
\begin{lemma} \label{lem1}
For any $T>0$, it holds that for any $ p \in [1, \infty]$, 
\begin{equation} \label{basicenergy}
\begin{aligned}
\sup_{0\le t\le T}(\|\rho\|_{L^p} + \|\sqrt{\rho}u\|_{L^2}^2 + \|\sqrt{\kappa(\rho)}\nabla\rho\|_{L^2}^2) +  \int_0^T \int |\nabla u|^2 dxds \le C.
\end{aligned}
\end{equation}
\end{lemma}

As mentioned in the Section 1, the main theorem will be proved by using a contradiction argument. Denote $ 0< T^* <\infty$ the maximal existence time for the strong solution to the initial and boundary value problem\eqref{NSK}-\eqref{bdy}. Suppose that \eqref{blowcriterion} were false, that is 
\begin{equation} \label{keyassume}
M_0 := \lim_{T \rightarrow T^*}(\|\nabla \rho\|_{L^{\infty}(0,T; W^{1,q})} + \| u\|_{L^s(0,T; L^r_{\omega})}) < \infty.
\end{equation}
Under the condition \eqref{keyassume}, one will extend the existence time of the strong solutions to \eqref{NSK}-\eqref{bdy} beyond $T^*$, which contradicts the definition of maximum of $T^*$.

The first key step is to derive the $L^2$-norm of the first order spatial derivatives of $u$ under the assumption of initial data and \eqref{keyassume}. Here we define the material derivative $\dot{u}:= u_t + u\cdot\nabla u$.

\begin{lemma} \label{lem2}
Under the condition \eqref{keyassume}, it holds that for any $0 < T < T^*$,
\begin{equation} \label{gradvelo}
\begin{aligned}
\sup_{ 0\le t \le T} \|\nabla u\|_{L^2}^2 + \int_0^T \|\sqrt{\rho} \dot{u}\|_{L^2}^2 dt  \le C.
\end{aligned}
\end{equation}
\end{lemma}

\begin{proof}
Multiplying the momentum equations $\eqref{NSK}_2$ by $u_t$, and integrating the resulting equations over $\Omega$, we have
\begin{equation} \label{lem2_1}
\begin{aligned}
&\int \rho |\dot{u}|^2 dx + \frac{d}{dt} \int \mu(\rho)|d|^2 dx \\
=& \int \rho \dot{u} \cdot(u\cdot\nabla u) dx - \int \mu'(\rho)u\cdot\nabla \rho|d|^2 dx + \int \kappa(\rho) \nabla \rho \otimes\nabla \rho : \nabla u_t dx \\
 =&\frac{d}{dt}\int \kappa(\rho) \nabla \rho \otimes\nabla \rho : \nabla u dx+\int \kappa'(\rho)(u\cdot\nabla \rho) \nabla \rho \otimes\nabla \rho : \nabla u dx \\
 & + \int \kappa(\rho) \nabla (u\cdot \nabla \rho) \otimes\nabla \rho : \nabla u dx + \int \rho \dot{u} \cdot(u\cdot\nabla u) dx - \int \mu'(\rho)u\cdot\nabla \rho|d|^2 dx \\
 =& \frac{d}{dt}\int \kappa(\rho) \nabla \rho \otimes\nabla \rho : \nabla u dx+ \sum_{k=1}^4 I_k. \\
\end{aligned}
\end{equation}
To complete the proof, we should bound the terms $I_1$ to $I_4$. First, for $I_1$, we use the assumption \eqref{keyassume} and apply H\"{o}lder inequality,
\begin{equation} \label{lem2_2}
\begin{aligned}
I_1 & = \int \kappa'(\rho)(u\cdot\nabla \rho) \nabla \rho \otimes\nabla \rho : \nabla u dx\\
& \le \|\kappa'(\rho)\|_{L^{\infty}}\|\nabla \rho\|_{L^6}^3 \|u\cdot\nabla u\|_{L^2} \\
& \le \|u\cdot\nabla u\|_{L^2}^2 +C.
\end{aligned}
\end{equation}
For $I_2$, we devide it into two parts, and simply use H\"{o}lder inequality to get
\begin{equation} \label{lem2_3}
\begin{aligned}
I_2 & = \int \kappa(\rho) \nabla (u\cdot\nabla \rho) \otimes\nabla \rho : \nabla u dx\\
& \le \|\kappa(\rho)\|_{L^{\infty}}\|\nabla \rho\|_{L^{\infty}}\|\nabla^2 \rho\|_{L^2} \|u\cdot\nabla u\|_{L^2} + \|\kappa(\rho)\|_{L^{\infty}}\|\nabla \rho\|_{L^{\infty}}^2 \|\nabla u\|_{L^2}^2\\
& \le C \|u\cdot\nabla u\|_{L^2}^2 +C(1+\|\nabla u\|_{L^2}^2).
\end{aligned}
\end{equation}
For $I_3$, using Cauchy-Schwarz inequality with $\epsilon$ to get
\begin{equation} \label{lem2_4}
\begin{aligned}
I_3 & =\int \rho \dot{u} \cdot(u\cdot\nabla u) dx \\
& \le \epsilon \|\sqrt{\rho}\dot{u}\|_{L^2}^2 + C(\epsilon) \|u\cdot\nabla u\|_{L^2}^2,\\
\end{aligned}
\end{equation}
and finally remark that $d = \frac{1}{2} (\nabla u + (\nabla u)^T)$, one has
\begin{equation} \label{lem2_5}
\begin{aligned}
I_4 & =\int \mu'(\rho)u\cdot\nabla \rho|d|^2 dx  \\
& \le \|\mu'(\rho)\|_{L^{\infty}}\|\nabla \rho\|_{L^{\infty}}\|\nabla u\|_{L^2}\|u\cdot\nabla u\|_{L^2}\\
& \le C\|\nabla u\|_{L^2}^2 + C\|u\cdot\nabla u\|_{L^2}^2.\\
\end{aligned}
\end{equation}
To obtain the second order spatial derivatives of the velocity $u$, we make good use of the Stokes type estimates on the momentum equations $\eqref{NSK}_2$ by simply put $F=-\rho \dot{u} - \mathrm{div}(\kappa(\rho)\nabla \rho \otimes \nabla \rho)$. Then applying Lemma \ref{stokesesti}, we derive that
\begin{equation} \label{lem2_6}
\begin{aligned}
\|\nabla u\|_{H^1} + \|P\|_{H^1} & \le C(1+ \|\nabla \rho\|_{L^{\infty}}) \|F\|_{L^2} \\
& \le C(1+ \|\nabla \rho\|_{L^{\infty}}) \|\rho \dot{u} + \mathrm{div}(\kappa(\rho)\nabla \rho \otimes \nabla \rho)\|_{L^2} \\
& \le C_*\|\sqrt{\rho}\dot{u}\|_{L^2} + C\|\nabla \rho\|_{L^6}^3 + C\|\nabla \rho\|_{L^{\infty}}\|\nabla^2 \rho\|_{L^2}\\
&  \le C_*\|\sqrt{\rho}\dot{u}\|_{L^2} + C,\\
\end{aligned}
\end{equation}
where $C_*$ is a positive number.

Now we substitute \eqref{lem2_2}-\eqref{lem2_5} into \eqref{lem2_1}, deduces that
\begin{equation} \label{lem2_7}
\begin{aligned}
&\int \rho |\dot{u}|^2 dx + \frac{d}{dt} \int \mu(\rho)|d|^2 dx \\
 \le & \frac{d}{dt}\int \kappa(\rho) \nabla \rho \otimes\nabla \rho : \nabla u dx+ \epsilon \|\sqrt{\rho}\dot{u}\|_{L^2}^2 + C(1+\|\nabla u\|_{L^2}^2)\\
 &+  C(\epsilon) \|u\cdot\nabla u\|_{L^2}^2\\
 \le & \frac{d}{dt}\int \kappa(\rho) \nabla \rho \otimes\nabla \rho : \nabla u dx+ \epsilon \|\sqrt{\rho}\dot{u}\|_{L^2}^2 + C(1+\|\nabla u\|_{L^2}^2) \\
 &+  \delta \|\nabla u\|_{H^1}^2 + C(\epsilon, \delta)(\|u\|_{L^r_{\omega}}^s+1)\|\nabla u\|_{L^2}^2\\
 \le & \frac{d}{dt}\int \kappa(\rho) \nabla \rho \otimes\nabla \rho : \nabla u dx+ \epsilon \|\sqrt{\rho}\dot{u}\|_{L^2}^2 + C(1+\|\nabla u\|_{L^2}^2) \\
 &+ C_* \delta \|\sqrt{\rho}\dot{u}\|_{L^2}^2 + C(\epsilon, \delta)(\|u\|_{L^r_{\omega}}^s+1)\|\nabla u\|_{L^2}^2,\\
\end{aligned}
\end{equation}
where we use Lemma \ref{weaksobolev} in the second inequality, and \eqref{lem2_6} is used to get the third one.
Then choosing $\epsilon, \delta$ small enough, we get
\begin{equation} \label{lem2_8}
\begin{aligned}
&\int \rho |\dot{u}|^2 dx + \frac{d}{dt} \int \mu(\rho)|d|^2 dx \\
 \le & \frac{d}{dt}\int \kappa(\rho) \nabla \rho \otimes\nabla \rho : \nabla u dx+ C(1+\|\nabla u\|_{L^2}^2)(\|u\|_{L^r_{\omega}}^s+1).\\
\end{aligned}
\end{equation}
By the assumption \eqref{keyassume} and Cauchy-Schwarz inequality, it is easily seen that
\begin{equation} \label{lem2_9}
\begin{aligned}
 C\int |\kappa(\rho)| |\nabla \rho \otimes\nabla \rho : \nabla u|dx \le  \frac{\underline{\mu}}{4} \|\nabla u\|_{L^2}^2 +C.
\end{aligned}
\end{equation}
Taking this into account, we can conclude from \eqref{lem2_8} and the Gronwall inequality that \eqref{gradvelo} holds for all $0 \le T < T^*$. Therefore we complete the proof of Lemma \ref{lem2}.
\end{proof}

To continue our proof, we will derive the estimate of $\sqrt{\rho}u_t$ by using the compatibility condition \eqref{initialcompa} on the initial data. More precisely, we have the following lemma.
\begin{lemma} \label{lem3}
Under the condition \eqref{keyassume}, it holds that for any $0 < T < T^*$,
\begin{equation} \label{rhou_t}
\begin{aligned}
\sup_{ 0\le t \le T} \|\sqrt{\rho} u_t \|_{L^2}^2 + \int_0^T \|\nabla u_t\|_{L^2}^2 dt  \le C.
\end{aligned}
\end{equation}
\end{lemma}

\begin{proof}
Differentiating the momentum equations $\eqref{NSK}_2$ with respect to $t$, along with the continuity equation $\eqref{NSK}_1$, we get
\begin{equation} \label{lem3_1}
\begin{aligned}
& \rho u_{tt} + \rho u\cdot\nabla u_t - \mathrm{div}(2\mu(\rho) d_t)  + \nabla P_t\\
= &  (u\cdot\nabla \rho)(u_t + u\cdot \nabla u) - \rho u_t\cdot \nabla u -  \mathrm{div}(2\mu'(\rho)(u\cdot\nabla \rho) d) \\
+ &  \mathrm{div}(\kappa'(\rho)(u\cdot\nabla \rho)\nabla \rho\otimes\nabla\rho) +2 \mathrm{div}(\kappa(\rho)\nabla (u\cdot\nabla \rho)\otimes\nabla\rho).\\
\end{aligned}
\end{equation}
Multiplying \eqref{lem3_1} by $u_t$ and integrating over $\Omega$, we get after integartion by parts that
\begin{equation} \label{lem3_2}
\begin{aligned}
&\frac{1}{2}\frac{d}{dt}\int \rho |u_t|^2 dx + 2 \int \mu(\rho)|d_t|^2 dx 
= \int -2 \rho u\cdot \nabla u_t\cdot u_t dx\\
+ &\int (u\cdot \nabla \rho)(u\cdot\nabla u)\cdot u_t dx - \int \rho u_t \cdot \nabla u\cdot u_t dx \\
+ & \int  2\mu'(\rho)(u\cdot\nabla \rho) d: \nabla u_t dx -  \int \kappa'(\rho)(u\cdot\nabla \rho)\nabla \rho\otimes\nabla\rho :\nabla u_t dx \\
 -& \int 2 \kappa(\rho)\nabla (u\cdot\nabla \rho)\otimes\nabla\rho:\nabla u_t dx =: \sum_{k=1}^6 J_k.\\
\end{aligned}
\end{equation}
To proceed, we estimate the terms from $J_1$ to $J_6$. First
\begin{equation} \label{lem3_3}
\begin{aligned}
J_1 & = \int -2 \rho u\cdot \nabla u_t\cdot u_t dx\\
& \le C\|\rho\|_{L^{\infty}}^{\frac{1}{2}}\|\sqrt{\rho}u_t\|_{L^3}\|u\|_{L^6}\|\nabla u_t\|_{L^2}\\
& \le C\|\sqrt{\rho}u_t\|_{L^{2}}^{\frac{1}{2}}\|\sqrt{\rho}u_t\|_{L^6}^{\frac{1}{2}}\|\nabla u\|_{L^2}\|\nabla u_t\|_{L^2}\\
& \le C\|\sqrt{\rho}u_t\|_{L^{2}}^{\frac{1}{2}}\|\nabla u\|_{L^2}\|\nabla u_t\|_{L^2}^{\frac{3}{2}}\\
& \le \frac{1}{12}\underline{\mu} \|\nabla u_t\|_{L^2}^{2} + C\|\sqrt{\rho}u_t\|_{L^{2}}^{2}\|\nabla u\|_{L^2}^{4}\\
& \le \frac{1}{12}\underline{\mu} \|\nabla u_t\|_{L^2}^{2} + C\|\sqrt{\rho}u_t\|_{L^{2}}^{2}.\\
\end{aligned}
\end{equation}
Similarly,
\begin{equation} \label{lem3_4}
\begin{aligned}
J_2 & = \int (u\cdot \nabla \rho)(u\cdot\nabla u)\cdot u_t dx \\
& \le C\|\nabla \rho\|_{L^{\infty}}\|\nabla u\|_{L^2}\|u\|_{L^6}^2\| u_t\|_{L^6}\\
& \le C\|\nabla \rho\|_{L^{\infty}}\|\nabla u\|_{L^2}^3 \|\nabla u_t\|_{L^2}\\
& \le \frac{1}{12}\underline{\mu}  \|\nabla u_t\|_{L^2}^{2} + C,\\
\end{aligned}
\end{equation}
\begin{equation} \label{lem3_5}
\begin{aligned}
J_3 & =  - \int \rho u_t \cdot \nabla u\cdot u_t dx \\
& \le C\| \rho\|_{L^{\infty}}^{\frac{1}{2}}\|u_t\|_{L^6}\|\sqrt{\rho}u_t\|_{L^3}\|\nabla u\|_{L^2}\\
& \le C\|\nabla u_t\|_{L^2}\|\sqrt{\rho}u_t\|_{L^2}^{\frac{1}{2}}\|\sqrt{\rho} u_t\|_{L^6}^{\frac{1}{2}}\\
& \le C\|\sqrt{\rho}u_t\|_{L^2}^{\frac{1}{2}}\|\nabla u_t\|_{L^2}^{\frac{3}{2}}\\
& \le \frac{1}{12}\underline{\mu} \|\nabla u_t\|_{L^2}^{2} + C\|\sqrt{\rho}u_t\|_{L^{2}}^{2},\\
\end{aligned}
\end{equation}
\begin{equation} \label{lem3_6}
\begin{aligned}
J_4 & = \int  2\mu'(\rho)(u\cdot\nabla \rho) d: \nabla u_t dx \\
& \le C\| \mu'(\rho)\|_{L^{\infty}}\|\nabla \rho\|_{L^{\infty}}\|u\|_{L^6}\|\nabla u\|_{L^3}\|\nabla u_t\|_{L^2}\\
& \le C\|\nabla u\|_{L^2}^{\frac{3}{2}}\|\nabla u\|_{H^1}^{\frac{1}{2}} \|\nabla u_t\|_{L^2}\\
& \le \frac{1}{12}\underline{\mu} \|\nabla u_t\|_{L^2}^{2} + C\|\nabla u\|_{H^1}^{2},\\
\end{aligned}
\end{equation}
\begin{equation} \label{lem3_7}
\begin{aligned}
J_5 & = \int \kappa'(\rho)(u\cdot\nabla \rho)\nabla \rho\otimes\nabla\rho :\nabla u_t dx \\
& \le C\| \kappa'(\rho)\|_{L^{\infty}}\|\nabla \rho\|_{L^{\infty}}^3\|u\|_{L^2}\|\nabla u_t\|_{L^2}\\
& \le\frac{1}{12}\underline{\mu} \|\nabla u_t\|_{L^2}^{2} + C,\\
\end{aligned}
\end{equation}
and finally remark that $ 3< q \le 6$, by the assumption \eqref{keyassume}, one has
\begin{equation} \label{lem3_8}
\begin{aligned}
J_6 & =\int 2 \kappa(\rho)\nabla (u\cdot\nabla \rho)\otimes\nabla\rho:\nabla u_t dx  \\
& \le C\| \kappa(\rho)\|_{L^{\infty}}\|\nabla \rho\|_{L^{\infty}}^2\|\nabla u\|_{L^2}\|\nabla u_t\|_{L^2}\\
& \quad + C\| \kappa(\rho)\|_{L^{\infty}}\|\nabla \rho\|_{L^{\infty}}\|\nabla^2 \rho\|_{L^3}\|u\|_{L^6}\|\nabla u_t\|_{L^2} \\
& \le \frac{1}{12}\underline{\mu}  \|\nabla u_t\|_{L^2}^{2} + C.\\
\end{aligned}
\end{equation}
It remains to estimate $\|\nabla u\|_{H^1}$ since it appears in the estimate of term $J_4$, see \eqref{lem3_6}.  Indeed, we can duduce from Lemma \ref{stokesesti} that
\begin{equation} \label{lem3_9}
\begin{aligned}
\|\nabla u\|_{H^1} +\|P\|_{H^1}& \le C(1+ \|\nabla \rho\|_{L^{\infty}}) \|F\|_{L^2} \\
& \le C(1+ \|\nabla \rho\|_{L^{\infty}}) \|\rho u_t + \rho u\cdot \nabla u + \mathrm{div}(\kappa(\rho)\nabla \rho \otimes \nabla \rho)\|_{L^2} \\
& \le C(\|\sqrt{\rho} u_t\|_{L^2} + \|u\|_{L^6}\|\nabla u\|_{L^3} + \|\nabla \rho\|_{L^6}^3 + \|\nabla \rho\|_{L^{\infty}}\|\nabla^2 \rho\|_{L^2})\\
&  \le C\|\sqrt{\rho}u_t\|_{L^2} + \frac{1}{2} \|\nabla u\|_{H^1} + C,\\
\end{aligned}
\end{equation}
which implies
\begin{equation} \label{lem3_10}
\begin{aligned}
\|\nabla u\|_{H^1}  \le C\|\sqrt{\rho}u_t\|_{L^2} +  C.\\
\end{aligned}
\end{equation}
Combining all the estimates \eqref{lem3_3}-\eqref{lem3_8} and \eqref{lem3_10}, we deduce that
\begin{equation} \label{lem3_11}
\begin{aligned}
&\frac{1}{2}\frac{d}{dt}\int \rho |u_t|^2 dx + 2 \int \mu(\rho)|d_t|^2 dx  \\
\le &  \frac{1}{2}\underline{\mu}  \|\nabla u_t\|_{L^2}^{2}  + C (1+\|\sqrt{\rho}u_t\|_{L^2}^2),  \\
\end{aligned}
\end{equation}
together with the fact that
$$ 2 \int |d_t|^2 dx = \int |\nabla u_t|^2 dx, $$
we obtain \eqref{rhou_t} by applying the Gronwall inequality. Therefore the proof of Lemma \ref{lem3} is completed.
\end{proof}

\begin{lemma} \label{lem4}
Under the condition \eqref{keyassume}, it holds that for any $0 < T < T^*$,
\begin{equation} \label{rhohighesti}
\begin{aligned}
\sup_{ 0\le t \le T} (\|\rho_t \|_{W^{1,q}}+\|u\|_{H^2} + \|P\|_{H^1})+ \int_0^T (\| u \|_{W^{2,q}}^2 + \|P\|_{W^{1,q}}^2) dt  \le C.
\end{aligned}
\end{equation}
\end{lemma}

\begin{proof}
As a direct consequence of Lemma \ref{lem3} and \eqref{lem3_10}, we can easily conclude that
\begin{equation} \label{lem4_1}
\sup_{0 \le t \le T} (\|u\|_{H^2} +\|P\|_{H^1}) \le C.
\end{equation}
And, by use of the continuity equation $\eqref{NSK}_1$, one deduces that
\begin{equation} \label{lem4_2}
\begin{aligned}
\|\rho_t\|_{W^{1,q}} & \le C(\|\rho_t\|_{L^q} + \|\nabla \rho_t\|_{L^q})\\
& \le C(\|u\cdot\nabla \rho\|_{L^q} + \|\nabla (u\cdot\nabla\rho)\|_{L^q})\\
& \le C(\|u\|_{L^{\infty}}\|\nabla \rho\|_{L^q} + \|u\|_{L^{\infty}}\|\nabla^2\rho\|_{L^q}+ \|\nabla u\|_{L^6}\|\nabla \rho\|_{L^{\frac{6q}{6-q}}} )\\
&  \le C\|u\|_{H^2}\|\nabla \rho\|_{W^{1,q}},\\
\end{aligned}
\end{equation}
by the assumption \eqref{keyassume} and \eqref{lem4_1}, the boundedness of $\|\rho_t\|_{W^{1,q}}$ is verified.\\
Finally, apply \eqref{stokesestir} in Lemma \ref{stokesesti} with $F = -\rho u_t - \rho u\cdot\nabla u - \mathrm{div}(\kappa(\rho)\nabla \rho\otimes\nabla\rho)$ to get
\begin{equation} \label{lem4_3}
\begin{aligned}
\|\nabla u\|_{W^{1,q}} + \|P\|_{W^{1,q}}& \le C(1+\|\nabla \rho\|_{L^{\infty}})(\|\rho u_t\|_{L^q} +\| \rho u\cdot\nabla u\|_{L^q} \\
&~~\quad+\|\kappa(\rho)|\nabla^2\rho||\nabla\rho|\|_{L^q}+\|\kappa'(\rho)|\nabla\rho|^3\|_{L^q})\\
& \le C(\|\rho u_t\|_{L^q} +\| \rho u\cdot\nabla u\|_{L^q} +1) \\
& \le C(\|\sqrt{\rho}u_t\|_{L^2}^{\frac{6-q}{2q}}\|\sqrt{\rho}u_t\|_{L^6}^{\frac{3q-6}{2q}}+\|\nabla u\|_{L^2}^{\frac{6(q-1)}{5q-6}}\|\nabla u\|_{W^{1,q}}^{\frac{4q-6}{5q-6}} + 1),\\
\end{aligned}
\end{equation}
by Young's inequality and Sobolev embedding inequality, it can be easily seen that
\begin{equation} \label{lem4_4}
\begin{aligned}
\|\nabla u\|_{W^{1,q}}^2 + \|P\|_{W^{1,q}} ^2 &\le C\|\sqrt{\rho}u_t\|_{L^2}^{\frac{6-q}{q}}\|\nabla u_t\|_{L^2}^{\frac{3(q-2)}{q}}+ C\|\nabla u\|_{L^2}^{\frac{12(q-1)}{q}} + C\\
& \le C\|\sqrt{\rho}u_t\|_{L^2}^{\frac{6-q}{q}}\|\nabla u_t\|_{L^2}^{\frac{3(q-2)}{q}}+  C.\\
\end{aligned}
\end{equation}
Hence
\begin{equation} \label{lem4_5}
\begin{aligned}
\int_0^T (\|\nabla u\|_{W^{1,q}}^2 + \|P\|_{W^{1,q}} ^2)dt & \le C\int_0^T  \|\sqrt{\rho}u_t\|_{L^2}^{\frac{6-q}{q}}\|\nabla u_t\|_{L^2}^{\frac{3(q-2)}{q}} dt +  C\\
& \le C(\sup_{0 \le t\le T} \|\sqrt{\rho}u_t\|_{L^2}^2)^{\frac{6-q}{2q}} \int_0^T \|\nabla u_t\|_{L^2}^2 dt +C \\
& \le C,
\end{aligned}
\end{equation}
here the second inequality holds since $ q \le 3$.
Therefore we complete the proof of Lemma \ref{lem4}.
\end{proof}
\begin{proof}[Proof of Theorem 1.3]
In fact, in view of \eqref{gradvelo} and \eqref{rhohighesti}, it is easy to see that the functions $(\rho,u)(x, t=T^*) = lim_{t \rightarrow T^*}(\rho,u)$ have the same regularities imposed on the initial data \eqref{initialreg} at the time $t = T^*$. Furthermore, 
\begin{equation*}
\begin{aligned}
&-\mathrm{div}(2\mu(\rho)d) + \nabla P + \mathrm{div}(\kappa(\rho)\nabla \rho\otimes\nabla \rho)|_{t= T^*} \\ 
&\quad = \lim_{t \rightarrow T^*} \rho^{\frac{1}{2}}(\rho^{\frac{1}{2}}u_t + \rho^{\frac{1}{2}}u\cdot\nabla u) := \rho^{\frac{1}{2}}g|_{t = T^*}\\
\end{aligned}
\end{equation*}
with $g = (\rho^{\frac{1}{2}}u_t +\rho^{\frac{1}{2}}u\cdot\nabla u) |_{t=T^*} \in L^2$ due to \eqref{rhou_t}. Thus the functions $(\rho, u)|_{t=T^*}$ satisfy
the compatibility condition \eqref{initialcompa} at time $T^*$. Therefore we can take $(\rho, u)|_{t=T^*}$ as the initial data and apply the local existence theorem (Theorem \ref{local}) to extend the local strong solution beyond $T^*$. This contradicts the definition of maximal existence time $T^*$, and thus, the proof of Theorem \ref{thm} is completed.
\end{proof}




\begin{thebibliography}{20}
\bibitem{CCK}
Y. Cho; H. J. Choe; H. Kim, \textit{Unique solvability of the initial boundary value problems for compressible viscous fluids}. J. Math. Pures Appl. (9)  83  (2004),  no. 2, 243-275. 
\bibitem{ChoY}
Y. Cho; H. Kim, \textit{Unique solvability for the density-dependent Navier-Stokes equations}. Nonlinear Anal. 59 (2004), no. 4, 465-489.
 \bibitem{ChoeK}
H. J. Choe; H. Kim, \textit{Strong solutions of the Navier-Stokes equations for nonhomogeneous incompressible fluids}. Comm. Partial Differential Equations 28 (2003), no. 5-6, 1183-1201.
\bibitem{Gra}
L. Grafakos, \textit{Classical Fourier analysis}. Second edition. Graduate Texts in Mathematics, 249. Springer, New York, 2008. xvi+489 pp.
\bibitem{HW1}
X. D. Huang; Y. Wang, \textit{Global strong solution to the 2D nonhomogeneous incompressible MHD system}. J. Differential Equations 254 (2013), no. 2, 511-527.
\bibitem{HW2}
X. D. Huang; Y. Wang, \textit{Global strong solution with vacuum to the two dimensional density-dependent Navier-Stokes system}. SIAM J. Math. Anal. 46 (2014), no. 3, 1771-1788.
\bibitem{HW3}
X. D. Huang; Y. Wang, \textit{Global strong solution of 3D inhomogeneous Navier-Stokes equations with density-dependent viscosity}. J. Differential Equations 259 (2015), no. 4, 1606-1627.
\bibitem{K}
H. Kim, \textit{A blow-up criterion for the nonhomogeneous incompressible Navier-Stokes equations}. SIAM J. Math. Anal.  37  (2006),  no. 5, 1417-1434.
\bibitem{LaSoN}
O. Ladyzhenskaya; V. A. Solonnikov; N. N. Uralceva, \textit{Linear and quasilinear equations of parabolic type}. Translations of Mathematical Monographs, Vol. 23 American Mathematical Society, Providence, R.I. 1968 xi+648 pp.
\bibitem{Li}
H. Y. Li, \textit{A blow-up criterion for the density-dependent Navier-Stokes-Korteweg equations in dimension two}, Acta Appl. Math. 2019.
\bibitem{TW}
Z. Tan; Y. J. Wang, \textit{Strong solutions for the incompressible fluid models of Korteweg type}. Acta Math. Sci. Ser. B Engl. Ed.  30  (2010),  no. 3, 799-809.
\bibitem{W}
T. Wang, \textit{Unique solvability for the density-dependent incompressible Navier-Stokes-Korteweg system}. J. Math. Anal. Appl.  455  (2017),  no. 1, 606-618.
\bibitem{XW}
X. Y.  Xu; J. W. Zhang, \textit{A blow-up criterion for 3D compressible magnetohydrodynamic equations with vacuum}. Math. Models Methods Appl. Sci.  22  (2012),  no. 2, 1150010, 23 pp. 

\end{thebibliography}
\end{document}